\documentclass[10pt]{amsart}
\usepackage{a4,amsmath,amssymb,amsthm,eucal}
\usepackage{epsfig}
\newcommand\N{\mathbb{N}}
\newcommand\R{\mathbb{R}}
\newcommand\Z{\mathbb{Z}}
\newcommand\Q{\mathbb{Q}}

\newcommand{\forget}[1]{}

\def\Om{{\Omega}}  %%%{{\bar{\Omega}}}
\def\om2{{\Om\times\Om}}
\def\M{{\mathcal M}}

\newcommand{\E}{\Sigma}

\def\supp{\mathrm{supp}\,}

\def\Lip{\mathrm{Lip}}

\newcommand{\res}{\llcorner} %   {\mathop{\hbox{
\newcommand{\MM}{\mathbb M}

\usepackage{color}
\newcommand{\weakto}{\rightharpoonup}

\newcommand{\ld}{[\![}
\newcommand{\rd}{]\!]}
\newtheorem{theorem}{Theorem}[section]

\newtheorem{lemma}[theorem]{Lemma}
\newtheorem{proposition}[theorem]{Proposition}
\newtheorem{corollary}[theorem]{Corollary}

\theoremstyle{remark}
\newtheorem{remark}[theorem]{Remark}

\numberwithin{equation}{section}

\title{Structure of metric cycles and normal one-dimensional currents}

\author{Emanuele Paolini}%
\address{Dipartimento di Matematica ``U.~Dini'', Universit\`{a} di
Firenze, viale Morgagni 67/A, 50134 Firenze, Italy.}

\author{Eugene Stepanov}%
\address{
St.Petersburg Branch
of the Steklov Mathematical Institute of the Russian Academy of Sciences,
Fontanka 27,
191023 St.Petersburg,
Russia
\and
Department of Mathematical Physics, Faculty of Mathematics and Mechanics,
St. Petersburg State University, Universitetskij pr.~28, Old Peterhof,
198504 St.Petersburg, Russia%, email: stepanov.eugene@gmail.com
}
\email{stepanov.eugene@gmail.com}
\thanks{The work of the second author was financed by GNAMPA, by RFBR grant \#11-01-00825,
  by the project 2008K7Z249 ``Trasporto ottimo di massa,
disuguaglianze geometriche e funzionali e applicazioni'' of the
Italian Ministry of Research,
 as well as by the project ANR-07-BLAN-0235 OTARIE}

\begin{document}

\begin{abstract}
We prove that every one-dimensional real Am\-bro\-sio-Kirchheim normal current
in a Polish (i.e.\ complete separable metric) space
 can be naturally represented as an integral of
 simpler currents associated to Lipschitz curves. As a consequence a representation of
 every such current with zero boundary (i.e.\ a cycle) as an integral
 of so-called elementary solenoids (which are, very roughly speaking, more or less the same as
 asymptotic cycles introduced by S.~Schwartzman)
 is obtained.
The latter result on cycles is in fact a  generalization of the analogous result
proven by S.~Smirnov for classical Whitney currents in a Euclidean space.
 The same results are true for every complete metric space under suitable set-theoretic assumptions.
\end{abstract}

\maketitle
%\tableofcontents

\section{Introduction}

In~\cite{PaoSte11-acycl} it has been shown that  every acyclic normal one-dimensional real current in a complete metric space
 can be naturally decomposed in curves, the decomposition preserving the mass and the boundary mass.
 Namely, roughly speaking, every such current $T$ can be represented as an integral
  \begin{align*}
T&=\int_{\Theta(E)} \ld\theta\rd \, d\eta(\theta)
\end{align*}
 of simple rectifiable currents $\ld\theta\rd$ associated to injective Lipschitz curves $\theta\colon [0,1]\to E$
 over some measure $\eta$ defined on the latter set of curves $\Theta(E)$,
  the mass of the current $\MM(T)$ being equal to the integral
 of the masses $\MM(\ld\theta\rd)$ (in this particular case equal to lengths $\ell(\theta)$)  of the respective curves,
 \begin{align*}
\MM(T)&=\int_{\Theta(E)} \MM(\ld\theta\rd)\, d\eta(\theta)=\int_{\Theta(E)} \ell(\theta)\, d\eta(\theta),
\end{align*}
with $\eta$-a.e.\ $\theta\in \Theta(E)$ belonging to the support of $T$, and a similar decomposition being valid also
 for boundary masses.
This is a direct generalization to metric currents introduced first by E.~De Giorgi and further
studied by L.~Ambrosio and B.~Kirchheim in~\cite{AmbrKirch00} of the analogous result for Whitney currents in a Euclidean space proven in~\cite{Smirnov94}.

The primary goal of this paper is to prove the analogous decomposition result for all (not only acyclic) real one-dimensional
metric currents. This is accomplished in Corollary~\ref{co_cycl1_repr_ac2a} based on Theorem~\ref{th_cycl1_repr_ac1} which
fills the gap by providing an appropriate decomposition of \emph{cycles}, i.e.\ real one-dimensional
metric currents without boundary. It is curious to mention that the latter theorem is mainly based on the decomposition of acyclic currents.

Once the primary goal is accomplished, it becomes natural to ask whether any cycle can be decomposed as an integral of currents associated to closed curves. Unfortunately, as shown in~\cite{Smirnov94} this is not true even for the Euclidean space, but at least in a Euclidean space
every
one-dimensional real Whitney currents with zero boundary (i.e.\ a cycle)
can be decomposed in so-called elementary \emph{solenoids} (called also solenoidal vector charges in~\cite{Smirnov94}).
Such solenoids, i.e.\ the natural ``elementary'' cycles, are strictly related to the \emph{asymptotic cycles}
introduced by S.~Schwartzman in~\cite{Schwartz57}
and further studied in~\cite{Schwartz97} (in fact, roughly speaking, up to technical details,
and in particular up to the fact that Schwartzman asymptotic cycles are normally
defined as elements of the space of homology classes~\cite{MunMar09sc},
one may identify the two notions). It is worth remarking that these objects appear quite natural
in the problem of representation of homology classes of manifolds
(see~\cite{MunMar09sc,MunMar11a,MunMar11b,MunMar09-II}).
The decomposition of a one-dimensional cycle into such solenoids appeared to be quite helpful
in the study of Mather's minimal measures~\cite{Bang99,DePasGelGran06}.

Here we prove the analogous result for Ambrosio-Kirchheim currents in an arbitrary complete metric space.
Namely, we introduce the notion of a solenoid
as a current $S$ over a metric space $E$ such that there exists
a
Lipschitz curve $\theta\colon \R\to E$ with $\Lip\, \theta \leq 1$
with the property
\[
S =\lim_{t\to +\infty} \frac{1}{2t}\ld \theta \res [-t,t]\rd
\]
in the appropriately weak sense,
while the trace $\theta(\R)$ of the curve $\theta$ is in the support of $S$, i.e.\
$\theta(\R)\subset \supp S$.
%where $\ld\sigma\rd$ stands for the current generated by a
%curve $\sigma$.
We show then that, roughly speaking, for every cycle $T$ with compact support %with $\partial T=0$
there is a measure $\eta$ concentrated over
the set $C$ of solenoids of unit mass such that
\begin{align*}
T&=\int_{C} S\, d\eta(S),\\
\MM(T)&=\int_{C} \MM(S)\, d\eta(S),
\end{align*}
and a similar result holds also for arbitrary cycles (not necessarily with compact support).
The result we provide
for cycles with compact support in an arbitrary metric space (Corollary~\ref{co_decompNorm2_cycl0b})
is the precise generalization of the  result
of~\cite{Smirnov94} on decomposition of cycles in a Euclidean space restricted to cycles with compact support, since
for Ambrosio-Kirchheim normal currents in compact subset of a Euclidean space the notion of mass
coincides with that of the usual Whitney currents.
The careful reader would observe that the result we provide for the general case of currents with possibly noncompact
support (Theorem~\ref{th_decompNorm2_cycl0a}) is ``almost like'' the respective general result in a Euclidean space setting from~\cite{Smirnov94}, the difference standing in the different definitions of mass for metric currents and for Whitney currents in a Euclidean space.

It is curious to note that
although the technique used to prove Theorem~\ref{th_cycl1_repr_ac1} which is the basis for all the results present in this
paper resembles the basic idea of~\cite{Smirnov94} of extending the space $E$ by an ``extra dimension'' and considering the appropriate extension of the original current $T$, the main line of the proof is in  a certain sense opposite to that used
in~\cite{Smirnov94}. Namely, here we use the representation result for acyclic currents from~\cite{PaoSte11-acycl}
as a starting point, while
in~\cite{Smirnov94} one does the contrary, i.e.\ first proves the decomposition result for cycles and then deduces the
respective results for acyclic currents from the latter. Therefore, since the proofs in~\cite{PaoSte11-acycl} do not depend on the
results of~\cite{Smirnov94}, we may consider also the results present in this paper independent on that of~\cite{Smirnov94} even
in the Euclidean setting.

\section{Notation and preliminaries}

The metric spaces are always in the sequel assumed to be
complete.
The parametric length of a Lipschitz curve $\theta\colon [a,b]\to E$  will be denoted
by $\ell(\theta)$.
The space of Lipschitz functions $\theta\colon [0,1]\to E$ equipped with uniform distance and factorized by
reparameterization will be denoted by $\Theta(E)$ (see~\cite{PaoSte11-acycl}). Every element of $\Theta(E)$ therefore
represents an oriented rectifiable curve. For a finite Borel measure $\eta$ over $\Theta(E)$ we set $\eta(i):=
e_{i\#}\eta$, where $e_i\colon \Theta(E)\to E$ are defined by $e_i(\theta):=\theta(i)$, $i=0,1$.

In the sequel we will always assume that the mass measures of the currents we are dealing with are all tight
(in fact, Radon, since the underlying metric space is complete).
This is
not restrictive because, as mentioned in~\cite{AmbrKirch00}, the theory of metric currents remains valid under such a requirement.
Thus, all our results hold in every complete metric space $E$ for normal currents $T$
when its mass measure $\mu_T$ (and the mass measure of its boundary $\mu_{\partial T}$, if appropriate) is tight, and hence,
in particular, for normal currents in a Polish (i.e.\ complete separable metric) space.
Equivalently, one could assume that
the density character (i.e.\ the minimum cardinality of a dense subset) of every metric space is an Ulam number.
This guarantees that every finite positive Borel measure is tight (even Radon when the space is complete),
is concentrated on some $\sigma$-compact subset and the support of this measure is separable (see, e.g., proposition~7.2.10
from~\cite{Bogachev06}), and is consistent with the Zermelo-Fraenkel set theory.

All the measures we will consider in the sequel are signed Borel measures
with finite total variation over some metric space $E$.
The narrow topology on measures is defined by duality with the space $C_b(E)$ of continuous bounded functions.
The supremum norm over $C_b(E)$ is denoted by $\|\cdot\|_\infty$.

For metric spaces $X$ and $Y$ we denote
by $\Lip(X,Y)$  (resp.\ $\Lip_k(X,Y)$ and $\Lip_b(X,Y)$) the set of all Lipschitz maps
(resp.\ all Lipschitz maps
with Lipschitz constant
$k$, the set of bounded Lipschitz maps)
$f\colon X\to Y$.
If $Y=\R$ we write just $\Lip(X)$, $\Lip_k(X)$, $\Lip_b(X)$ respectively.

For the metric currents we use the notation from~\cite{PaoSte11-acycl} which is almost completely
taken from~\cite{AmbrKirch00}, except mainly the notation for the mass measure.
In particular, $D^k(E)=\Lip_b(E)\times (\Lip (E))^k$ stands for the space of
metric $k$-forms, its elements (i.e.\ $k$-forms) being denoted
by $f\,d\pi$, where $f\in \Lip_b(E)$, $\pi\in (\Lip(E))^k$,
 $\M_k(E)$ stands for the space of $k$-dimensional metric currents,
 $\mathcal{N}_k(E)$ stands for the space of $k$-dimensional normal metric currents,
$\MM(T)$ stands for the mass of a current $T$, and $\mu_T$ stands for the mass measure
associated to this current.  The one-dimensional current associated to a
Lipschitz curve $\theta\colon [a,b]\to E$  will be denoted
by $\ld\theta\rd$, namely,
\[
\ld\theta\rd(f\,d\pi):=\int_a^b f(\theta(t))\, d\pi(\theta(t))
\]
for every $f\, d\pi \in D^1(E)$. Recall that $\MM(\ld\theta\rd)\leq \ell(\theta)$.
The weak topology in $\M_k(E)$ is defined by a family of seminorms
$\{ T\mapsto |T(\omega)|\,:\, \omega\in D^k(E)\}$. It is clearly a Hausdorff locally convex topology.
The notation $S\leq T$ means that $S$ is a subcurrent of $T$ in the sense that
$\MM(S)+\MM(T-S)=\MM(T)$.

\section{Decomposition of normal currents in curves}

The first important result of this paper is the following statement.

\begin{theorem}\label{th_cycl1_repr_ac1}
Let $T\in \M_1(E)$ satisfy $\partial T=0$.
Then there is a finite positive Borel measure $\bar\eta$ over $\Theta(E)$ such that
\begin{align*}
T(\omega)&=\int_{\Theta(E)} \ld\theta\rd (\omega)\, d\bar\eta(\theta),\\
\MM(T)&=\int_{\Theta(E)} \MM(\ld\theta\rd)\, d\bar\eta(\theta)=\int_{\Theta(E)} \ell(\theta)\, d\bar\eta(\theta),
\end{align*}
for all $\omega\in D^1(E)$, while $\bar\eta(0)=\bar\eta(1)=\mu_T$ and
$\bar\eta$-a.e.\ $\theta\in \Theta(E)$ belongs to $\supp T$
and has $\MM(\ld\theta\rd)=\ell(\theta)=1$.
\end{theorem}

To prove this theorem we need some preliminary constructions.
Equip the space $E\times [0,1]$ with the distance
\[
d_\infty((u_1,t_1), (u_2,t_2)):= d(u_1,u_2)\vee \lvert t_1-t_2\rvert.
\]
Let $T\in \M_1(E)$ satisfy $\partial T=0$.
Define
\[
T':=T\times \mu_{\ld [0,1]\rd} + \mu_T\times \ld [0,1]\rd\in \mathcal{N}_1(E\times [0,1]).
\]
Letting $P_E\colon (x,t)\in E\times [0,1]\mapsto x\in E$ and $P\colon (x,t)\in E\times [0,1]\to t\in [0,1]$, we get
\[
P_{E\#}T'=T, \qquad P_{\#} T'= \MM(T) \ld [0,1]\rd.
\]
Further,
\[
\partial T'=\mu_T\otimes (\delta_1-\delta_0)
\]
and $\MM(T')=\MM(T)$ by Lemma~\ref{lm_cycl1_Mprod1}.
At last, we have the following statement.

\begin{lemma}\label{lm_cycl1_proj1}
$T'$ is acyclic.
\end{lemma}

\begin{proof}
Let $C' \leq T'$ be a cycle, and let $C:=P_{\#}C'$.
%, where
%$P$ stands for the projection to $[0,1]$.
We have
\begin{align*}
\MM(P_{\#} T' - C) + \MM(C) & = \MM(P_{\#}(T'-C')) + \MM(P_{\#}C') \\
& \leq \MM(T'-C')+\MM(C') = \MM(T')=\MM(P_{\#}T'),
\end{align*}
which means $C\leq P_{\#}T'= \MM(T) \ld [0,1]\rd$, and hence $C=0$.
But
$\MM(C)=\MM(C')$, since otherwise in the above relationship the inequality would be strict, which is impossible.
Hence, $\MM(C')=0$, i.e. $C'=0$.
\end{proof}

We are now ready to prove the announced result.

\begin{proof}[Proof of Theorem~\ref{th_cycl1_repr_ac1}]
By representation theorem for acyclic one-dimensional currents~\cite{PaoSte11-acycl}[theorem~5.1] one has
\begin{align*}
T'(\omega')&=\int_{\Theta(E\times [0,1])} \ld\theta\rd (\omega)\, d\eta'(\theta),\\
\MM(T')&=\int_{\Theta(E\times [0,1])} \ell(\theta)\, d\eta'(\theta),
\end{align*}
for some finite positive Borel measure $\eta'$ over $\Theta(E\times [0,1])$ and
for all $\omega'\in D^1(E\times [0,1])$, while $\eta'$-a.e.\ $\theta\in E\times [0,1]$ is an
arc belonging to $\supp T'$,
and
\[
\eta'(0)=\mu_T\otimes \delta_0,\qquad
\eta'(1)=\mu_T\otimes \delta_1.
\]
Denoting $\bar\eta:=P_{E\#}\eta'$, we get
\begin{equation}\label{eq_cycl1_reprMT1}
\begin{aligned}
T(\omega) &=(P_{E\#}T')(\omega)=
T'(\omega\circ P_E)  =
\int_{\Theta(E\times [0,1])} \ld\theta\rd (\omega\circ P_E)\, d\eta'(\theta),\\
&= \int_{\Theta(E\times [0,1])} \ld P_E(\theta)\rd (\omega)\, d\eta'(\theta)=
\int_{\Theta(E)} \ld\theta\rd (\omega)\, d\bar\eta(\theta).
\end{aligned}
\end{equation}
This also implies
\[
\MM(T)\leq \int_{\Theta(E)} \MM(\ld\theta\rd)\, d\bar\eta(\theta) \leq \int_{\Theta(E)} \ell(\theta)\, d\bar\eta(\theta),
\]
On the other hand,
\begin{equation}\label{eq_cycl1_reprMT1b}
\begin{aligned}
\int_{\Theta(E)} \ell(\theta)\, d\bar\eta(\theta) & =
\int_{\Theta(E\times [0,1])} \ell(P_E(\theta))\, d\eta'(\theta) \\
&\leq \int_{\Theta(E\times [0,1])} \ell(\theta)\, d\eta'(\theta) = \MM(T')=\MM(T),
\end{aligned}
\end{equation}
and hence
\begin{equation}\label{eq_cycl1_reprMT2}
\MM(T)=\int_{\Theta(E)} \MM(\ld\theta\rd)\, d\bar\eta(\theta)=\int_{\Theta(E)} \ell(\theta)\, d\bar\eta(\theta).
\end{equation}
Further, in~\eqref{eq_cycl1_reprMT1b} the inequality is in fact an equality, and hence
\[
\ell(P_E(\theta))=\ell(\theta)\geq d_\infty(\theta(0),\theta(1))\geq 1
\]
 for $\eta'$-a.e.\ $\theta\in \Theta(E\times [0,1])$,
the latter inequality being true because $P(\theta(0))=0$ and $P(\theta(1))=1$ for $\eta'$-a.e.\ $\theta\in \Theta(E\times [0,1])$.
Thus $\ell(\theta)\geq 1$ for $\bar\eta$-a.e.\ $\theta\in \Theta(E)$.
But then from~\eqref{eq_cycl1_reprMT2} one has
$\MM(T)\geq \bar\eta(\Theta(E))$.
Recall now that $\bar\eta(0)= \bar\eta(1)=\mu_T$. This implies
$\MM(T)\leq \bar\eta(\Theta(E))$, and therefore $\MM(T)= \bar\eta(\Theta(E))$.
Thus $\ell(\theta)=1$ for $\bar\eta$-a.e.\ $\theta\in \Theta(E)$.

The relationship~\eqref{eq_cycl1_reprMT1} implies also
\begin{align*}
(\partial T)(f)&=\int_{\Theta(E)} \left( f(\theta(1))-f(\theta(0))\right)\, d\bar\eta(\theta)\\
& = \int_{E} f(x)\, d\bar\eta(1)(x) - \int_{E} f(x)\, d\bar\eta(0)(x)
=\int_E f(x)\, d
(\bar\eta(1)-\bar\eta(0))(x),
\end{align*}
so that $\partial T=\bar\eta(1)-\bar\eta(0)$,
which gives $\bar\eta(0)=\bar\eta(1)$.
Finally,  $\bar\eta$-a.e.\
$\theta\in \Theta(E)$ belongs to $\supp P_{E\#}T'=\supp T$.
\end{proof}

The above theorem allows to formulate the following corollary on the structure of all
one-dimensional real metric currents.

\begin{corollary}\label{co_cycl1_repr_ac2a}
Let $T\in \mathcal{N}_1(E)$.
Then there is a finite positive Borel measure $\bar\eta$ over $\Theta(E)$
with the total mass $\bar\eta(\Theta(E)) \leq \MM(T)+\MM(\partial T)$
such that
\begin{align*}
T(\omega)&=\int_{\Theta(E)} \ld\theta\rd (\omega)\, d\bar\eta(\theta),\\
\MM(T)&=\int_{\Theta(E)} \MM(\ld\theta\rd)\, d\bar\eta(\theta)=\int_{\Theta(E)} \ell(\theta)\, d\bar\eta(\theta),
\end{align*}
for all $\omega\in D^1(E)$, with $\bar\eta$-a.e.\ $\theta\in \Theta(E)$ belonging to $\supp T$.
\end{corollary}

\begin{proof}
Decompose (say, by proposition~3.8 from~\cite{PaoSte11-acycl})
$T=S+C$ with $S\leq T$ acyclic and $C\leq T$ a cycle, i.e.\ $\partial C=0$.
Use theorem~5.1 from~\cite{PaoSte11-acycl} to decompose $S$ in curves and the above
Theorem~\ref{th_cycl1_repr_ac1} to do the same for $C$. This gives the result.
\end{proof}

As a toy application we mention here for purely illustrative purposes the following immediate corollary on nonexistence of nontrivial normal currents in the space without rectifiable curves.

\begin{corollary}\label{co_cycl1_noNk}
Let $E$ be a metric space which has no nonconstant Lipschitz curves.
Then $\mathcal{N}_k(E)$ contains only the zero current for all $k\geq 1$.
\end{corollary}

\begin{proof}
For $k=1$ this follows from the Corollary~\ref{co_cycl1_repr_ac2a}. For general $k$ proceed by induction: suppose that
the statement is true for $k-1$, i.e.\ $\mathcal{N}_{k-1}(E)$ contains only the zero current. Let $T\in \mathcal{N}_k(E)$, and
consider an arbitrary $\pi_k\in \Lip(E)$. Then for every $t\in\R$ the slice
$\langle T,\pi_k, t \rangle\in \mathcal{N}_{k-1}(E)$ and hence $\langle T,\pi_k, t \rangle=0$ by induction assumption, which
by slicing theorem~5.6 from~\cite{AmbrKirch00} gives $T\res d\pi_k=0$. Consider now an arbitrary
$f\,d\pi\in D^{k-1}(E)$.
Then minding the alternating property of currents (theorem~3.5 from~\cite{AmbrKirch00}), we get
\[
|T(f\, d\pi_1\wedge \ldots\wedge d\pi_{k-1}\wedge d\pi_k)|= |T\res d\pi_k (f\, d\pi)|=0,
\]
so that $T=0$.
\end{proof}

\section{Decomposition of cycles in solenoids}

This section is dedicated to another decomposition result for one-dimensional real metric currents without boundaries
(i.e.\ cycles). In fact, given the validity of Theorem~\ref{th_cycl1_repr_ac1}, it is
natural to ask whether any cycle can be decomposed as an integral of (currents associated to) \emph{closed} curves. As it is shown in~\cite{Smirnov94} this is unfortunately not true even in the $3$-dimensional Euclidean space $\R^3$ but at least in all Euclidean spaces there is a natural decomposition of cycles in so-called \emph{solenoids}
(called also solenoidal vector charges in~\cite{Smirnov94}). We will extend this result to generic metric spaces.

We start with the following corollary of Theorem~\ref{th_cycl1_repr_ac1}.

\begin{corollary}\label{co_cycl1_repr_ac2}
Let $T\in \M_1(E)$ satisfy $\partial T=0$.
There is a finite positive  Borel measure $\tilde\eta$ over $X:= C([0,1];E)$ (with the topology of uniform
convergence)
concentrated over $\Lip_1([0,1];E)$ such that
\begin{align*}
T(\omega)&=\int_{X} \ld\theta\rd (\omega)\, d\tilde\eta(\theta),\\
\MM(T)&=\int_{X} \ell(\theta)\, d\tilde\eta(\theta)=\tilde\eta(X),
\end{align*}
for all $\omega\in D^1(E)$, while $\tilde\eta$-a.e.\ $\theta\in X$ belongs to $\supp T$,
and $\tilde\eta(0)=\tilde\eta(1)$.
\end{corollary}

\begin{proof}
Let $h\colon \Theta(E)\to X$ send every $\theta\in \Theta(E)$ in its parameterization with constant speed.
It is enough to set then $\tilde \eta:= h_{\#}\bar\eta$, where $\bar\eta$ is provided by Theorem~\ref{th_cycl1_repr_ac1}.
\end{proof}

Now we will prove the following extension statement.

\begin{proposition}\label{prop_cycl1_repr_acR}
Let $\tilde\eta$ be a Borel measure over $C([0,1];E)$ satisfying the properties provided
by Corollary~\ref{co_cycl1_repr_ac2}. Then there is a
Borel measure $\hat\eta$ over $C(\R;E)$ (equipped with the topology of uniform convergence
over bounded intervals)
concentrated over $\Lip_1(\R;E)$ such that
\begin{itemize}
  \item[(a)] $\pi_{\#} \hat\eta=\tilde\eta$, where $\pi\colon C(\R;E)\to C([0,1];E)$ is the map defined by
  $\pi (\theta):= \theta\res[0,1]$,
  \item[(b)] $g^\pm_{\#} \hat \eta=\hat\eta$, where $g^\pm\colon C(\R;E)\to C(\R;E)$ are the shift maps defined by
  $g^\pm (\theta)(t):= \theta(t\pm 1)$,
    \item[(c)] for $\hat\eta$-a.e.\ $\theta\in C(\R;E)$ one has $\theta(\R)\subset \supp T$.
\end{itemize}
\end{proposition}

\begin{remark}\label{rem_cycl1_repr_acR1}
The measure $\hat\eta$ provided by the above Proposition~\ref{prop_cycl1_repr_acR} satisfies
\begin{align*}
(m-n) T(\omega) &=
\sum_{i=n}^{m-1} \int_{C([0,1];E)} \ld\theta \rd(\omega)\,d\tilde\eta(\theta) =
\sum_{i=n}^{m-1} \int_{\Lip_1(\R;E)} \ld\pi(\theta) \rd(\omega)\,d\hat\eta(\theta) \\
& =
\sum_{i=n}^{m-1} \int_{\Lip_1(\R;E)} \ld\theta\res [0,1] \rd(\omega)\,d\hat\eta(\theta) \\
& =  \sum_{i=n}^{m-1} \int_{\Lip_1(\R;E)} \ld\theta\res [i, i+1] \rd(\omega)\,d\hat\eta(\theta)\\
& =   \int_{\Lip_1(\R;E)} \ld\theta\res [n, m] \rd(\omega)\,d\hat\eta(\theta)
\end{align*}
for all $\{m, n\}\subset \Z$ and $\omega\in D^1(E)$.
Analogously,
\begin{align*}
(m-n) \MM(T) &
=   \int_{\Lip_1(\R;E)} \MM(\ld\theta\res [n, m] \rd)\,d\hat\eta(\theta).
\end{align*}
\end{remark}

\begin{proof}%[Proof of Proposition~\ref{prop_cycl1_repr_acR}]
The proof will be achieved in two steps.

{\sc Step 1.}
Without loss of generality we may assume $\tilde \eta$ to be a probability measure.
Let $X:= C([0,1];E)$ be equipped with the usual uniform topology, and let
$e_t\colon X\to E$ be defined by
$e_t(\theta):=\theta(t)$.  Consider the Borel probability measures $\eta_x^\pm$ over $X$
defined by the disintegration formulae
\begin{align*}
\tilde \eta &= (e_{0\#}\tilde \eta)\otimes \eta_x^+= \tilde \eta(0)\otimes \eta_x^+,\\
\tilde \eta &= (e_{1\#}\tilde \eta)\otimes \eta_x^-= \tilde \eta(1)\otimes \eta_x^-,
\end{align*}
i.e.\
\begin{align*}
\tilde \eta(e) &= \int_E\eta_x^+(e)\,d\tilde \eta(0)(x)
%\tilde \eta(e) &
= \int_E\eta_x^-(e)\,d\tilde \eta(1)(x)
\end{align*}
for every Borel $e\subset X$.
It is worth remarking that since $\tilde\eta(0)=\tilde\eta(1)$, then $\eta_x^+=\eta_x^-$, while
both measures are defined for $\tilde\eta(0)=\tilde\eta(1)$-a.e.\ $x\in E$, so we may omit
the superscripts writing just $\eta_x$ instead of $\eta_x^\pm$.

Define now inductively the measures $\eta_k$ over $X^k$ by setting
for all $k\in \N$ and Borel $e\subset X^k$,
\begin{align*}
\eta_1 &:= \tilde \eta,\\
% \eta_k (e) &:= \int_{X} \eta^+_{\theta_{k-1}(1)}(e\cap (\{(\theta_1,\ldots, \theta_{k-1})\}\times X))
\eta_k (e) &:= \int_{X} \eta^+_{\theta_{k-1}(1)}(e_{(\theta_1,\ldots, \theta_{k-1})})
\,d\eta_{k-1}(\theta_1,\ldots, \theta_{k-1}),
\end{align*}
where $e_{(\theta_1,\ldots,\theta_{k-1})} := \{\theta \in X \colon (\theta_1,\ldots,\theta_{k-1},\theta) \in e\}$,
so that in particular,
\begin{align*}
\eta_k (e_1\times\ldots\times e_k) &= \int_{e_1\times\ldots\times e_{k-1}} \eta^+_{\theta_{k-1}(1)}(e_k)\,d\eta_{k-1}(\theta_1,\ldots, \theta_{k-1}),\\
\eta_2 (e_1\times e_2) &= \int_{e_1} \eta^+_{\theta_1(1)}(e_2)\, d\eta_1(\theta_1).
\end{align*}
%(of course, this defines $\eta_k$ over rectangles, and then we extend $\eta_k$ over
%the whole Borel $\sigma$-algebra of $X^k$ by the standard extension procedure).
Let $\pi_{k-1}\colon X^k= X^{k-1}\times X\to X^{k-1}$ and
$\pi^{k-1}\colon X^k= X\times X^{k-1}\to X^{k-1}$ be defined by
\begin{align*}
    \pi_{k-1} (x_1,\ldots, x_{k-1}, x_k) & :=(x_1,\ldots, x_{k-1}),\\
    \pi^{k-1} (x_1, x_2,\ldots,  x_k) & := (x_2,\ldots, x_k),
\end{align*}
i.e.\ as in Lemma~\ref{lm_cycl1_probext1}.
Note that
\begin{align*}
\eta_k(e\times X) = \int_{e} \eta^+_{\theta_{k-1}(1)}(X) \, d\eta_{k-1}(\theta_1,\ldots, \theta_{k-1})
 & =\int_{e} \,d\eta_{k-1}(\theta_1,\ldots, \theta_{k-1}) \\
 &=
 \eta_{k-1} (e)
\end{align*}
for every Borel $e\subset X^{k-1}$,
which means
$\pi_{k-1\#} \eta_k = \eta_{k-1}$ for all $k\in \N$.
On the other hand,
\begin{align*}
\eta_2(X\times e_2) & = \int_{X\times e_2} \eta^+_{\theta(1)}(e_2)\, d\eta(\theta)=\int_{E} \eta_x^+ (e_2)\, d\eta(1)(x)
=\eta(e_2)=
 \eta_1 (e_2).
\end{align*}
Assuming inductively that $\pi^{k-1}_{\#} \eta_k = \eta_{k-1}$ for some $k\in \N$, $k\geq 2$, we get
\begin{align*}
\eta_{k+1}(X\times e_2\times\ldots\times e_{k+1}) & = \int_{X\times e_2\times\ldots\times e_{k+1}} \eta^+_{\theta_k(1)}(e_{k+1}) \,d\eta_k(\theta_1,\ldots, \theta_k)\\
 &=\int_{e_2\times\ldots\times e_{k+1}}\,  \eta^+_{\theta_k(1)}(e_{k+1})\,  d\eta_{k-1}(\theta_2,\ldots, \theta_k) \\
 &=
 \eta_k (e_2\times\ldots\times e_{k+1}),
\end{align*}
and hence by induction
$\pi^{k-1}_{\#} \eta_k = \eta_{k-1}$ for all $k\in \N$, $k\geq 2$.

By Lemma~\ref{lm_cycl1_probext1} there is a Borel measure $\eta_*$ over $X^\Z$
such that for $p_j(x):= (x)_j$ one has
\begin{equation}\label{eq_etastar1}
\eta_*\left(
\bigcap_{j=k}^l p_j^{-1} (e_j)
\right)= \eta_{l-k}
\left(
\prod_{j=k}^l e_j
\right)
\end{equation}
for $p_j(x)=x_j$.
In particular,
one has
\begin{itemize}
\item[(i)] $\eta_*$ is concentrated over $(\Lip_1([0,1];E))^\Z$,
\item[(ii)]
$\bar\pi_{\#} \eta_*=\tilde\eta$, where $\bar\pi\colon X^\Z\to X^\Z$ is the map defined by
  $\bar\pi (\bar\theta):= \bar\theta_1$,
\item[(iii)] $\eta_*$ is invariant with respect to the shift
maps $\bar g^\pm \colon X^\Z\to X^\Z$ defined by
  $(\bar g^\pm (\bar \theta))_k:= \theta_{k\pm 1}$,
\item[(iv)] for $\eta_*$-a.e.\ $\bar\theta = \{\theta_k\}_{k\in \Z}$ one has
$\theta_k(1)=\theta_{k+1}(0)$ for all $k\in \Z$,
\item[(v)] for $\eta_*$-a.e.\ $\bar\theta = \{\theta_k\}_{k\in \Z}$ one has $\theta_k\subset \supp T$
for all $k\in \Z$.
\end{itemize}
In fact,~(ii) and~(iii) are immediate from~\eqref{eq_etastar1},~(i) follows from the fact that $\tilde\eta$ is concentrated over $\Lip_1([0,1];E)$, while to prove~(iv) it is enough, in view of~(iii), to prove that
\[
\eta_2(\{(\theta_1,\theta_2)\in X^2\colon \theta_1(1)\neq \theta_2(0)\})=0.
\]
The latter equality follows by a simple calculation
\begin{align*}
    \eta_2(\{(\theta_1,\theta_2)\in X^2\colon \theta_1(1)\neq \theta_2(0)\}) & =
    \int_X \eta^+_{\theta_1(1)}(\{\theta_2\in X\colon \theta_1(1)\neq \theta_2(0)\})\,d\tilde\eta(\theta_1)\\
    &= \int_E \eta^+_{x}(\{\theta_2\in X\colon x\neq \theta_2(0)\})\,d\tilde\eta(1)(x)=0,
\end{align*}
the final equality being due to the fact that $\eta^+_x$ is concentrated over
$e_0^{-1}(x)$.
Finally,
\begin{align*}
\eta_*(\{\bar\theta\colon \theta_k\not\in \supp T\}) & =
\eta_*(\{\bar\theta\colon \theta_1\not\in \supp T\})  \mbox{ by~(iii)}\\
& =
\tilde\eta(\{\theta\colon \theta\not\in \supp T\})  \mbox{ by~(ii)}\\
& =0,
\end{align*}
which proves~(v).

{\sc Step 2.}
Define
the map $q\colon X^{\Z}\to %L^0(\R;E)
E^{\R}$ by setting
\[
q(\bar\theta)(t):= \theta_{\lfloor t\rfloor} (\{t\}),
\]
and set $\hat \eta:= q_{\#}\eta_*$.
Clearly,
for $\eta_*$-a.e.\ $\bar\theta = \{\theta_k\}_{k\in \Z}$ one has
$q(\bar\theta)\in C(\R;E)$, since for every $k\in \Z$ one has
\begin{align*}
\lim_{t\to k-0} q(\bar\theta)(t) &= %\lim_{t\to k^-}\theta_{\lfloor t\rfloor} (\{t\}) = %Ok, ma si abbrevia
\theta_{k-1} (1) = \theta_k(0), \qquad \mbox{ by~(iv)}\\
& =\lim_{t\to k+0} q(\bar\theta)(t),
\end{align*}
and hence, by~(i), $q(\bar\theta)\in\Lip_1 (\R;E)$, so that
$\hat\eta$ is concentrated over $\Lip_1 (\R;E)$. Finally,~(a) follows from~(ii),~(b) follows from~(iii)
and~(c) follows from~(v).
\end{proof}

We are now in a position to prove our main results regarding decomposition of cycles in solenoids.
Let $E$ be a metric space with distance $d$ and $X\subset E$ be a $\sigma$-compact set.
Consider $\bar X$ to be equipped with the distance $d$ and consider the new distance $\tilde d$ over $\bar X$ provided by Lemma~\ref{lm_cycl1_sepremetr1} (with $\bar X$ in place of
$E$) and
let $\tilde X$ stand for the completion of $\bar X$ with respect to $\tilde d$ (equipped with $\tilde d$).
We may write, slightly abusing the notation, $\Lip(\tilde X)\subset \Lip(\bar X)$ identifying each $u\in \Lip(\tilde X)$ with its restriction to
$\bar X$. Analogously, $C_b(\tilde X)\subset C_b(\bar X)$. Thus, $D^k(\tilde X)\subset D^k(\bar X)$, and hence,
$\M_k(\bar X)\hookrightarrow \M_k(\tilde X)$ with continuous immersion (namely, for every $T\in \M_k(\bar X)$ one has
that $\tilde{\MM}(T)\leq \MM(T)$, where $\tilde{\MM}$ stands for the mass norm in $\M_k(\tilde X)$).
We start with the following general result.

\begin{theorem}\label{th_decompNorm2_cycl0a}
For every $T\in \mathcal{N}_1(E)$
having $\partial T=0$
and supported over a $\sigma$-compact set $X\subset E$
there is a finite positive Borel measure $\hat\eta$ over
$\Lip_1(\R;E)$
%$\Lip_1(\R;\tilde X)$
such that
for $\hat\eta$-a.e.\ $\theta$
there is a limit
\begin{equation}\label{eq_cycl1_Sthet1}
S_\theta =\lim_{t\to +\infty} \frac{1}{2t}\ld \theta \res [-t,t]\rd \in \M_1(\tilde X)
\end{equation}
in the weak sense of currents in $\M_1(\tilde X)$,
while
\begin{equation}\label{eq_cycl1_reprT1a}
\begin{aligned}
T(\omega)&=%\int_{\Lip_1(\R;\tilde X)}
\int_{\Lip_1(\R;E)}
S_\theta(\omega)\, d\hat\eta(\theta)\mbox{ for all }\omega\in D^1(\tilde X),\\
\MM(T)&=%\int_{\Lip_1(\R;\tilde X)} \tilde{\MM}(S_\theta)\, d\hat\eta(\theta) =
\hat\eta(\Lip_1(\R;E)) %(\Lip_1(\R;\tilde X)),
\end{aligned}
\end{equation}
and
%$\tilde{\MM}(S_\theta)=1$
%and
the trace $\theta(\R)\subset \supp T$,
\end{theorem}

\begin{proof}
For every $\omega\in D^1(E)$
and every $\theta\in \Lip_1(\R;E)$
 we define
\[
f_\omega(\theta):= \ld\theta\res [0,1]\rd(\omega).
\]
By Remark~\ref{rem_cycl1_repr_acR1} (with $m:=1$, $n:=0$) one has that
$f_\omega\in L^1(C(\R;E); \hat\eta)$ (so that in particular
$f_\omega$ is finite on $C(\R;E)$ for
$\hat\eta$-a.e. $\theta\in C(\R;E)$).
By the ergodic theorem one has the existence for $\hat\eta$-a.e.\ $\theta\in C(\R;E)$ of a limit
\begin{equation}\label{eq_cycl1_erg0}
\begin{aligned}
\bar f_\omega(\theta) &:= \lim_{k\to +\infty}\frac{1}{2k}\sum_{j=-k}^k f_\omega((g^+)^j(\theta))\\
& =\lim_{k\to +\infty, k\in \N}\frac{1}{2k}\ld\theta\res [-k,k]\rd(\omega),
\end{aligned}
\end{equation}
and the validity of the relationship
\begin{equation}\label{eq_cycl1_erg1}
\int_{C(\R;E)} f_\omega(\theta)\,d\hat\eta(\theta) = \int_{C(\R;E)} \bar f_\omega(\theta)\,d\hat\eta(\theta).
\end{equation}
Let $\{\omega^j\}\subset D^1(\tilde X)$ be as in the proof of Lemma~\ref{lm_cycl1_wkmetr}, and let
$C_j\subset C(\R;E)$ be such a set of curves that~\eqref{eq_cycl1_erg0} %and~\eqref{eq_cycl1_erg1} are
is valid
for $\omega=\omega^j$ and all $\theta\in C(\R;E)\setminus C_j$, so that $\hat\eta(C_j)=0$. Set
$C:=\cup_j C_j$.
Minding that
$\mu_{\partial \frac 1 {2k} \ld\theta\res [-k,k]\rd}(E)= 1/2k\to 0$
as $k\to \infty$ and $\MM\left(\frac 1 {2k} \ld\theta\res [-k,k]\rd\right)\leq 1$,
for all $\theta \in \Lip_1(\R;E)$, hence for $\hat\eta$-a.e.\
 $\theta$, while the $\tilde X$ is compact,
we get
that the sequence of currents $\{\frac{1}{2k}\ld\theta\res [-k,k]\rd\}$ is precompact
in the weak topology of currents in $\M_1(\tilde X)$.
On the other hand, by the choice of $C$ one has that
the latter sequence of currents %$\{\frac{1}{2k}\ld\theta\res [-k,k]\rd\}$
is convergent
in the distance $d_w$, and thus, by Lemma~\ref{lm_cycl1_wkmetr},
also in the weak sense of currents in $\M_1(\tilde X)$ for all $\theta \in C(\R;E)\setminus C$.

We have proven therefore the existence for $\hat\eta$-a.e.\ $\theta\in C(\R;E)$
of a limit
\[
S_\theta =\lim_{k\to +\infty, k\in \N} \frac{1}{2k}\ld \theta \res [-k,k]\rd
\]
in the weak sense of currents $\M_1(\tilde X)$
with
\begin{align*}
\int_{C(\R;E)} S_\theta(\omega)\,d\hat\eta(\theta) &=
\int_{C(\R;E)} f_\omega(\theta)\,d\hat\eta(\theta) =\int_{C(\R;E)} \ld\pi(\theta)\rd(\omega)\,d\hat\eta(\theta)\\
& = \int_{C([0,1];E)} \ld\sigma\rd(\omega)\,d\tilde\eta(\sigma)= T(\omega)
\end{align*}
for all $\omega\in D^1(\tilde X)$.
We show now that $S_\theta$ is in fact as in the statement being proven, i.e.\
\[
S_\theta =\lim_{k\to +\infty} \frac{1}{2t_k}\ld \theta \res [-t_k,t_k]\rd
\]
in the weak sense of currents
for every sequence $t_k\to +\infty$ as $k\to +\infty$, because
\begin{align*}
\frac{1}{2t_k} &  \ld \theta \res [-t_k,t_k]\rd   -
\frac{1}{2\lfloor t_k \rfloor}\ld \theta \res [-\lfloor t_k \rfloor,\lfloor t_k \rfloor]\rd\\
& =
\frac{1}{2t_k} \left(\ld \theta \res [-t_k,-\lfloor t_k \rfloor]\rd + \ld \theta \res [\lfloor t_k \rfloor, t_k]\rd\right)
+
 \frac{1}{2\lfloor t_k \rfloor} \left(1-\frac{\lfloor t_k \rfloor}{t_k} \right) \ld \theta \res [-\lfloor t_k \rfloor,\lfloor t_k \rfloor]\rd,
\end{align*}
which implies
\begin{align*}
\MM & \left(\frac{1}{2t_k}  \ld \theta \res [-t_k,t_k]\rd -
\frac{1}{2\lfloor t_k \rfloor}\ld \theta \res [-\lfloor t_k \rfloor,\lfloor t_k \rfloor]\rd\right) \leq
 \frac{2}{2t_k} +
\left(1-\frac{\lfloor t_k \rfloor}{t_k} \right) \to 0
\end{align*}
as $k\to \infty$.

Minding that
\[
\MM(T) = \hat\eta(C(\R;E)),
\]
and that $\hat\eta$ is concentrated over $\Lip_1(\R;E)$, we conclude the proof.
\end{proof}

We may now formulate the following important corollaries to the above statement.

\begin{corollary}\label{co_decompNorm2_cycl0b}
Let $E$ be a metric space.
Then for every $T\in \mathcal{N}_1(E)$ with compact support
having $\partial T=0$
there is a finite positive Borel measure $\eta$ over
$\Lip_1(\R;E)$ such that
for $\eta$-a.e.\ $\theta$
there is a limit
\[
%\begin{equation}\label{eq_cycl1_Sthet1a}
S_\theta =\lim_{t\to +\infty} \frac{1}{2t}\ld \theta \res [-t,t]\rd \in \M_1(E)
%\end{equation}
\]
in the weak sense of currents in $\M_1(E)$,
and the trace $\theta(\R)\subset \supp T$,
while
\begin{equation}\label{eq_cycl1_reprT1b}
\begin{aligned}
T&=\int_{\Lip_1(\R;E)} S_\theta\, d\eta(\theta),\\
\MM(T)&=\int_{\Lip_1(\R;E)} \MM(S_\theta)\, d\eta(\theta) =\eta(\Lip_1(\R;E)),
\end{aligned}
\end{equation}
so that in particular, $S_\theta\subset \M_1(E)$ has unit mass for $\eta$-a.e.\ $\theta\in \Lip_1(\R;E)$.
Finally, we may assume $\theta(\R)\subset \supp S_\theta$ for $\eta$-a.e.\ $\theta$.
\end{corollary}

\begin{proof}
Without loss of generality we assume $E$ to be compact.
We now repeat the proof of Theorem~\ref{th_decompNorm2_cycl0a} with the original space $E$
instead of the compactification $\tilde X$, getting the existence
for $\hat\eta$-a.e.\ $\theta$
of a limit
\[
%\begin{equation}\label{eq_cycl1_Sthet1a}
S_\theta =\lim_{t\to +\infty} \frac{1}{2t}\ld \theta \res [-t,t]\rd \in \M_1(E)
%\end{equation}
\]
in the weak sense of currents in $\M_1(E)$,
such that
\[
T(\omega)=\int_{\Lip_1(\R;E)} S_\theta(\omega)\, d\hat\eta(\theta)
\]
for all $\omega\in D^1(E)$, which implies
\[
%\begin{equation}\label{eq_cycl1_erg3}
\MM(T) \leq  \int_{C(\R;E)} \MM(S_\theta)\,d\hat\eta(\theta)\leq \hat\eta(C(\R;E)) =\MM(T),
%\end{equation}
\]
and hence in particular $\MM(S_\theta)=1$ for $\hat\eta$-a.e.\ $\theta\in C(\R;E)$.

The trace $\theta(\R)\subset \supp T$ for $\hat\eta$-a.e.\ $\theta\in C(\R;E)$ by Corollary~\ref{co_cycl1_repr_ac2}.
This gives all the claims of the theorem being proven but the last one for $\eta:=\hat\eta$. Finally, to prove the last claim, consider the set
\[
\Sigma:=\{S\in \M_1(E)\,:\, \partial S=0, \MM(S)\leq 1\}.
\]
Clearly $\Sigma$ is a convex compact subset of $\M_1(E)$, and $\Sigma$ equipped with the weak topology of currents
is compact and metrizable by Lemma~\ref{lm_cycl1_wkmetr}.
We claim now that if $S\in \Sigma$ is extremal, then $S=S_\theta$ for some $\theta\in \Lip_1(\R;E)$.
In fact, consider the representation
\begin{equation}\label{eq_cycl1_diracsol1}
\begin{aligned}
S(\omega)&=\int_{C(\R;E)} S_\theta(\omega)\, d\eta(\theta),\\
\MM(S)&=\int_{C(\R;E)} \MM(S_\theta)\, d\eta(\theta)=\eta(C(\R;E))
\end{aligned}
\end{equation}
for all $\omega\in D^1(E)$.
Note that~\eqref{eq_cycl1_diracsol1} implies that
for every Borel $e\subset C(\R;E)$, defined
\[
S_1(\omega):=\int_{e} S_\theta(\omega)\, d\eta(\theta),
\]
for all $\omega\in D^1(E)$,
one has $S_1\leq S$, because
\begin{align*}
(S-S_1)(\omega)&:=\int_{C(\R;E)\setminus e} S_\theta(\omega)\, d\eta(\theta),
\end{align*}
and hence
\begin{align*}
\MM(S_1)&\leq \eta(e)\\
\MM(S-S_1)&\leq \eta(C(\R;E)\setminus e),
\end{align*}
so that $\MM(S_1)+\MM(S-S_1)\leq \eta(C(\R;E))=\MM(S)$.
Since $S$ is extremal, then $S_1=\lambda S$ for some $\lambda \in [0,1]$ and thus $\MM(S_1)=\eta(e)=\lambda$, so that we can write
\begin{equation}\label{eq_cycl1_Som1}
S(\omega) =\frac{1}{\eta(e)}\int_e S_\theta(\omega)\, d\eta(\theta).
\end{equation}
If $S\neq S_\theta$, then there are two different curves $\{\theta_1,\theta_2\}\subset \Lip_1(\R;E)$
such that $R_1:=S_{\theta_1}\neq R_2:= S_{\theta_2}$ and
for every $\varepsilon>0$ one has $\eta(\hat B_\varepsilon(R_i))>0$, where
$\hat B_\varepsilon(R_i)$ stands for the set of $\theta\in C(\R;E)$ in the support
of $\eta$ such that $S_\theta\in B_\varepsilon(R_i)$, the notation
$B_\varepsilon(R_i)$ standing
for the ball of radius $\varepsilon$ and center $R_i$ in the space of cycles (with respect to the
%mass norm
distance $d_w$ provided by Lemma~\ref{lm_cycl1_wkmetr}), $i=1,2$.
Choose an $\omega\in D^1(E)$ such that
\[
\alpha:= |R_1(\omega)-R_2(\omega)| >0,
\]
and an $\varepsilon>0$ such that
\[
|R(\omega)-R_i(\omega)| <\frac \alpha 4 \mbox{ for all } R\in B_\varepsilon(R_i), \qquad i=1,2.
\]
Then
\begin{align*}
\Big| \frac{1}{\eta(\hat B_\varepsilon(R_i))} &\int_{\hat B_\varepsilon(R_i)} S_\theta(\omega)\, d\eta(\theta)  - R_i(\omega)
\Big| \\
& \leq
\frac{1}{\eta(\hat B_\varepsilon(R_i))}\int_{\hat B_\varepsilon(R_i)} |S_\theta(\omega)- R_i(\omega)|\, d\eta(\theta)  <\frac \alpha 4,\qquad i=1,2,
\end{align*}
so that
\[
\Big| \frac{1}{\eta(\hat B_\varepsilon(R_1))}\int_{\hat B_\varepsilon(R_1)} S_\theta(\omega)\, d\eta(\theta)   -
\frac{1}{\eta(\hat B_\varepsilon(R_2))}\int_{\hat B_\varepsilon(R_2)} S_\theta(\omega)\, d\eta(\theta)  \Big| \geq \frac \alpha 2.
\]
This contradicts the equality
\[
\frac{1}{\eta(\hat B_\varepsilon(R_1))}\int_{\hat B_\varepsilon(R_1)} S_\theta(\omega)\, d\eta(\theta)   =
\frac{1}{\eta(\hat B_\varepsilon(R_2))}\int_{\hat B_\varepsilon(R_2)} S_\theta(\omega)\, d\eta(\theta) =S(\omega)
\]
valid in view of~\eqref{eq_cycl1_Som1},
and thus shows the claim.

Clearly also for every extremal point $S$ of $\Sigma$ one has $\MM(S)=1$, hence
\[
\eta(C(\R;E))=\MM(S)=1,
\]
and therefore
we have proven  that for such $S$  one has the representation~\eqref{eq_cycl1_diracsol1}
with $S=S_\theta$ for $\eta$-a.e.\ $\theta\in C(\R;E)$. Since it has already been proven that one may assume
in~\eqref{eq_cycl1_diracsol1} that
$\theta(\R)\subset \supp S$ for $\eta$-a.e.\ $\theta\in C(\R;E)$, then  one has
 $\theta(\R)\subset \supp S_\theta$.
It remains now to refer
to Choquet theorem~\cite[theorem~4.2]{BishDeLeeuw59} to show the existence of a representation~\eqref{eq_cycl1_reprMT1b}
with $\theta(\R)\subset\supp S_\theta$ for $\eta$-a.e. $\theta\in \Lip_1(\R;E)$.
\end{proof}

Another corollary refers to the noncompact case.

\begin{corollary}\label{co_decompNorm2_cycl1b}
For every $T\in \mathcal{N}_1(E)$
having $\partial T=0$
and supported over a $\sigma$-compact set $X\subset E$
there is a finite positive Borel measure $\eta$ over
$\Lip_1(\R;\tilde X)$ such that
for $\eta$-a.e.\ $\theta$
there is a limit
\[
%\begin{equation}\label{eq_cycl1_Sthet1b}
S_\theta =\lim_{t\to +\infty} \frac{1}{2t}\ld \theta \res [-t,t]\rd \in \M_1(\tilde X)
%\end{equation}
\]
in the weak sense of currents in $\M_1(\tilde X)$,
and the trace $\theta(\R)\subset \supp T$,
while
\begin{equation}\label{eq_cycl1_reprT1}
\begin{aligned}
T&=\int_{\Lip_1(\R;\tilde X)} S_\theta\, d\eta(\theta),\\
\tilde{\MM}(T)&=\int_{\Lip_1(\R;\tilde X)} \tilde{\MM}(S_\theta)\, d\eta(\theta) =\eta(\Lip_1(\R;\tilde X)),
\end{aligned}
\end{equation}
so that in particular, $S_\theta\subset \M_1(\tilde X)$ has unit mass for $\eta$-a.e.\ $\theta\in \Lip_1(\R;\tilde X)$.
Finally, $\mu_{S_\theta}$ are concentrated over $\theta(\R)$ for $\eta$-a.e.\ $\theta$.
\end{corollary}

\begin{proof}
It is enough to apply Corollary~\ref{co_decompNorm2_cycl0b} with $\tilde X$ instead of $E$.
\end{proof}

\appendix

\section{Some statements regarding currents}

Here we collect some statements regarding currents which are used in this paper.
We start with the following statement regarding metrizability of the weak topology of currents.

\begin{lemma}\label{lm_cycl1_wkmetr}
Let $X\subset E$ be a $\sigma$-compact set.
Then there is a distance $d_w$ over  $\M_k(\bar X)$ %$\Sigma$
which generates a topology coarser than the weak
topology of currents, such that for every  $\Sigma\subset \M_k(\bar X)$ %\M_k(E)
weakly sequentially precompact,
the topology generated by $d_w$ over $\Sigma$
coincides with the weak one.

In particular, if $\Sigma\subset \M_k(E)$ is such that
the family of measures $\{\mu_T+\mu_{\partial T}\}_{T\in \Sigma}$ is uniformly tight
and there is a $C>0$ such that $\MM(T)+\MM(\partial T)\leq C$ for all $T\in\Sigma$, then
weak topology of currents is metrizable over $\Sigma$.
\end{lemma}

\begin{proof}
Let $\{K_\nu\}$ be an increasing sequence of compact subsets of $E$ such that
$X=\cup_\nu K_\nu$. Notice that $\Lip_m(K_\nu)$ is separable with respect to the norm
$\|\cdot \|_\infty$.
Recall that every function in $\Lip_m(X)$ can be uniquely extended to a function
in $\Lip_m(\bar X)$. Hence it is possible to
endow $\Lip_m(\bar X)$ %$\Lip_m(\bar X)=\Lip_m(X)$
with a separable metric inducing uniform convergence
on each $K_\nu$. Let $Z^m\subset \Lip_m(\bar X)$ and
\[
Z_b^{m,n}\subset \{u\in  \Lip_m(\bar X)\,:\, \|u\|_\infty \leq n\}
\]
 be countable dense subsets with respect to this metric. Set $Z:=\cup_{m=1}^\infty Z^m$
 and $Z_b:=\cup_{m=1, n=1}^\infty Z_b^{m,n}$.
We let then
\[
d_w(T, T') :=\sum_{j=1}^\infty 2^{-j} (|T(\omega^j)-T'(\omega^j)|\wedge 1),
\]
where $\{\omega^j\}= Z_b\times (Z)^k\subset D^k(\bar X)$, i.e.\
$\omega^j= f^j\, d\pi_1^j\wedge\ldots\wedge d\pi_k^j$
with $f^j\in Z_b$, $\pi_i^j\in Z$ for all $j\in \N$ and all $i=1,\ldots,k$.
To show that this is a distance, assume $d_w(T, T')=0$ for some $T\in \mathcal{N}_k(\bar X)$,
$T'\in \mathcal{N}_k(\bar X)$. This means $T(\omega^j)=T'(\omega^j)$ for all
$\omega^j$. But for any $\omega=f\, d\pi_1\wedge\ldots\wedge d\pi_k \in D^k(\bar X)$, letting $m\in \Z$ be such that
 $\Lip\, \pi_i\leq m$, $\Lip f\leq m$,
 we may find a sequence of
$\omega^j=f^j\, d\pi_1^j\wedge\ldots\wedge d\pi_k^j \in Z_b\times (Z)^k$ with
$f^j\to f$, $\pi_i^j\to \pi_i$, $i=1, \ldots, k$, pointwise over $X$ (in fact, even uniformly over each $K_\nu$) as $j\to \infty$
and  $\Lip\, \pi_i^j\leq m$, $\Lip f^j\leq m$ for all $j\in \N$ and $i=1, \ldots, k$.
Then for every $\bar x\in \bar X$  and and arbitrary $x\in X$ one has
\begin{align*}
|\pi_i^j(\bar x)-\pi_i(\bar x)| & \leq |\pi_i^j(\bar x)- \pi_i^j(x)| + |\pi_i^j(x)- \pi_i(x)| + |\pi_i(\bar x)- \pi_i(x)| \\
& \leq 2m d(\bar x, x) + |\pi_i^j(x)- \pi_i(x)|,
\end{align*}
which, minding the convergence $\pi_i^j(x)\to \pi_i(x)$, gives the convergence $\pi_i^j(\bar x)\to \pi_i(\bar x)$ as $j\to \infty$.
Hence, $\pi_i^j\to \pi_i$, $i=1, \ldots, k$ (and analogously $f^j\to f$) pointwise over $\bar X$, and
thus by the continuity property of currents
we get $T(\omega)=T'(\omega)$ for all
$\omega\in D^k(\bar X)$, which means $T= T'$.
Clearly, the topology induced by $d_w$ is coarser
than the weak topology of currents.

Let now $T_\nu\in \Sigma$, where $\Sigma$ be as in the statement of the lemma being proven, and assume
that $d_w(T_\nu,\bar T)\to 0$ for
some $\bar T\in \M_k(\bar X)$, so that $T_\nu(\omega^j)\to \bar T(\omega^j)$
for each $\omega^j\in Z_b\times (Z)^k$ as
$\nu\to \infty$. Under the assumptions on $\Sigma$,
every subsequence
of $T_\nu$ has a further subsequence (all subsequences not relabeled) such that
 $T_\nu\rightharpoonup T$ in the weak sense of currents, hence
 also $d_w(T_\nu, T)\to 0$
  as $\nu\to \infty$. Hence $d_w(\bar T, T)=0$, i.e.\ $\bar T=T$, and therefore
the whole sequence $\{T_\nu\}$ converges to $T$ in the weak sense of currents.

In the particular case indicated in the statement we
let $\{K_\nu\}$ be an increasing sequence of compact subsets of $E$ such that
$\mu_T(K_\nu)+\mu_{\partial T}(K_\nu)\leq 1/\nu$, $X:=\cup_\nu K_\nu$, and
refer to the fact that $\Sigma$ is sequentially precompact
in the weak topology of currents by theorem~5.2 from~\cite{AmbrKirch00}.
\end{proof}

\begin{remark}\label{rm_cycl1_dw1}
In the above Lemma~\ref{lm_cycl1_wkmetr} it is possible to choose the distance $d_w$ over  $\M_k(\bar X)$ so as to have
additionally the semicontinuity property for masses
\begin{equation}\label{eq_cycl1_semicont1}
\MM(T\res U)\leq\liminf_\nu \MM(T_\nu\res U)
\end{equation}
for every open $U\subset \bar X$, and in particular
\[
\MM(T)\leq\liminf_\nu \MM(T_\nu)
\]
whenever $d_w(T_\nu, T)\to 0$.
In fact, for this purpose let $\{x_i\}\subset \bar X$ stand for a countable dense subset of $\bar X$, and consider
the countable family of open sets $\mathcal{F}=\{U_j\}$ consisting of all finite unions of open balls $B_{r_j}(x_i)$, where
$\{r_j\}=\Q$ is the enumeration of rational numbers.  Let also $p_k$ stand for the projection map from $\R^k$
to the Euclidean unit ball $B_1(0)\subset \R^k$.
Now, in the proof of the above Lemma~\ref{lm_cycl1_wkmass}
when constructing $Z_b$ one should first add to $\tilde Z_b:=\cup_{m=1, n=1}^\infty Z_b^{m,n}$
to each $k$-uple of functions $(f_1,\ldots, f_k)\subset \tilde Z_b$ also
the function $p_k(f_1(\cdot),\ldots, f_k(\cdot))$, and then add to the obtained set of functions
(let us call it $Z_b'$)
all functions of the form $u 1_{U_j}$ for all $u\in Z_b'$ and $U_j\in \mathcal{F}$
thus forming the set $Z_b$.
 Now, to prove~\eqref{eq_cycl1_semicont1} it is enough to prove
\begin{equation}\label{eq_cycl1_semicont2}
 \sum_{i=1}^k T(f_i\,d\pi_i)\leq \liminf_\nu \MM(T_\nu\res U)
\end{equation}
 whenever $\sum_{i=1}^k f_i \leq 1_U$ and $\Lip \pi_i\leq 1$.
 One can then find
 sequences
 $\{f_i^j\}_{j=1}^\infty\subset Z_b$ and $\{\pi_i^\nu\}_{j=1}^\infty\subset Z$
such that  $\sum_{i=1}^k f_i^j \leq 1_U$,
$\Lip \pi_i^j\leq 1$ and
\[
f_i^j\to f_i, \qquad \pi_i^j\to \pi_i
\]
pointwise as $j\to \infty$.
Then
\begin{align*}
\liminf_\nu \MM(T_\nu\res U) & \geq \liminf_\nu \sum_{i=1}^k T_\nu (f_i^j\,d\pi_i^j) \\
& \geq \sum_{i=1}^k \liminf_\nu  T_\nu (f_i^j\,d\pi_i^j) = \sum_{i=1}^k T (f_i^j\,d\pi_i^j),
\end{align*}
and taking a limit in the above inequality as $j\to \infty$ we get~\eqref{eq_cycl1_semicont2}, and hence~\eqref{eq_cycl1_semicont1}.
\end{remark}

\begin{lemma}\label{lm_cycl1_wkmass}
If $T_\nu\in \M_k(E)$ and $T\in \M_k(E)$ be such that
$T_\nu\weakto T$ in the weak sense of currents and
$\MM(T_\nu)\to \MM(T)$
  as $\nu\to \infty$, then $\mu_{T_\nu}\weakto \mu_T$ in the narrow sense of measures.
\end{lemma}

\begin{proof}
One has $\mu_{T_\nu}(E)\to \mu(E)$ and
\[
\mu_T(U)\le \liminf_\nu \mu_{T_\nu}(U)
\]
for every open $U\subset E$, and therefore $\mu_{T_\nu}\rightharpoonup \mu_T$ in the narrow sense of measures
by theorem~8.2.3 from~\cite{Bogachev06}. The uniform tightness of $\{\mu_{T_\nu}\}$ follows then from Prokhorov theorem
for nonnegative measures
(theorem~8.6.4 from~\cite{Bogachev06}).
\end{proof}

\begin{remark}\label{rm_cycl1_dw2}
It is easy to observe that the result of the above Lemma~\ref{lm_cycl1_wkmass}
remains true if the condition $T_j\weakto T$ in the weak sense of currents is substituted by the weaker one
$d_w(T_j,T)\to 0$ once the distance $d_w$ satisfies the semicontinuity property~\eqref{eq_cycl1_semicont1}
(the proof is word-to-word identical to the above one).
\end{remark}

The following lemma allows to pass to diagonal subsequences in the weak convergence of currents.

\begin{lemma}\label{lm_currdiag1}
Let $E$ be a metric space, $T\in \mathcal{N}_k(E)$, $T_j\in \mathcal{N}_k(E)$ and $T_j^m\in \mathcal{N}_k(E)$ be such that
\begin{align*}
T_j\weakto T & \mbox { as } j\to \infty,\\
T_j^m\weakto T_j & \mbox { as } m\to \infty,
\end{align*}
in the weak sense of currents,
and
\begin{align*}
\MM(T_j)\to \MM(T), &  \qquad\MM(\partial T_j)\to \MM(\partial T) \mbox { as } j\to \infty,\\
\MM(T_j^m)\to \MM(T_j), & \qquad \MM(\partial T_j^m)\to \MM(\partial T_j)  \mbox { as } k\to \infty.
\end{align*}
Then  there is a subsequence of $m=m(j)$ such that
$T_j^{m(j)}\weakto T$ in the weak sense of currents,
$\mu_{T_j^{m(j)}}\weakto \mu_{T}$ and $\mu_{\partial T_j^{m(j)}}\weakto \mu_{\partial T}$ in the narrow
sense of measures as $j\to \infty$.
\end{lemma}

\begin{proof}
Note that under the conditions of the statement being proven
\begin{align*}
\mu_{T_j}\weakto \mu_T, \qquad\mu_{\partial T_j}\weakto \mu_{\partial T}&  \mbox { as } j\to \infty,\\
\mu_{T_j^m}\weakto \mu_{T_j}, \qquad \mu_{\partial T_j^m}\weakto \mu_{\partial T_j}&  \mbox { as } m\to \infty,
\end{align*}
in the narrow
sense of measures by Lemma~\ref{lm_cycl1_wkmass}.

Let $K_\nu\subset E$ and $K_\nu^j\subset E$ be such compact sets that
\begin{align*}
\mu_{T_j}(K_\nu^c)+ \mu_{\partial T_j}(K_\nu^c) & \leq 1/\nu \mbox{ for all }j\in \N,\\
\mu_{T_j^m}((K_\nu^j)^c)+ \mu_{\partial T_j^m}((K_\nu^j)^c) &\leq 1/\nu,
\mbox{ for all } m\in \N.
\end{align*}
Note that setting
\[
X:=\bigcup_{j,\nu} K_\nu^{j} \cup \bigcup_{\nu} K_\nu,
\]
we have that all $T_j$, $T_j^m$ and $T$ are concentrated over $\bar X$.
Let $d_w$ stand for the distance over $\mathcal{N}_k(\bar X)$ provided by Lemma~\ref{lm_cycl1_wkmetr}, and denote
by $\|\cdot\|_0$ the Kantorovich-Rubinstein norm metrizing the narrow topology on positive
finite Borel measures over $\bar X$ (see~\cite{Bogachev06}[theorem~8.3.2]).
For every $n\in \N$ choose a $j=j(n)$ and $m=m(n)$
such that
\begin{align*}
d_w(T_j,T)  \leq \frac 1 n &, d_w(T_j^m, T_j)\leq \frac 1 n, \\
\left\| \mu_{T_j}-\mu_T\right\|_0 \leq \frac 1 n, & \left\| \mu_{\partial T_j}-\mu_{\partial T}\right\|_0\leq \frac 1 n,\\
\left\| \mu_{T_j^m}-\mu_{T_j}\right\|_0 \leq \frac 1 n, & \left\| \mu_{\partial T_j^m}-\mu_{\partial T_j}\right\|_0\leq \frac 1 n.
\end{align*}
Clearly, with this construction
\begin{equation}\label{eq_cycl1_convdw2}
\begin{aligned}
T=\lim_{j\to \infty} T_j&= \lim_{j\to\infty} T_j^{m(j)}
\end{aligned}
\end{equation}
in distance $d_w$.
But the sequences $\{\mu_{T_j^{m(j)}}\}$ and $\{\mu_{\partial T_j^{m(j)}}\}$
converge
in the norm $\|\cdot\|_0$ (hence also in the narrow sense of measures), and therefore, they are uniformly tight by
the Prokhorov theorem
for nonnegative measures
(theorem~8.6.4 from~\cite{Bogachev06}).
Thus, by Lemma~\ref{lm_cycl1_wkmetr}, the convergence in~\eqref{eq_cycl1_convdw2} is also in the weak topology of currents.
\end{proof}

The following lemma is in fact implicitly contained in~\cite{PaoSte11-acycl} in the sense that its arguments are widely used in that paper. We make it explicit here for the readers' convenience.

\begin{lemma}\label{lm_currapprMAP1}
Let $E$ be a Banach space
with metric approximation property,
$T\in \mathcal{N}_k(E)$ with $\mu_T$ and $\mu_{\partial T}$ concentrated over a $\sigma$-compact set.
Then there is a sequence of currents $T_n\in \mathcal{N}_k(E_n)$ supported over
some finite dimensional subspaces
$E_n\subset E$, such that
$T_n\rightharpoonup T$ weakly as currents in $\M_k(E)$,
$\mu_{T_n}\rightharpoonup \mu_T$ and
$\mu_{\partial T_n}\rightharpoonup \mu_{\partial T}$ in the narrow sense of measures
as $n\to \infty$.
In particular, if $k=1$, then identifying the zero-dimensional currents with measures
one has $(\partial T_n)^\pm \rightharpoonup (\partial T)^\pm$ in the narrow sense of measures
as $n\to \infty$.
\end{lemma}

\begin{remark}\label{rm_currapprMAP1a}
From the proof of the above Lemma it is clear that when $T$ is a cycle (i.e.\ $\partial T=0$) with bounded support, then
$T_n$ are cycles as well.
\end{remark}

\begin{proof}
Let
$\{K_\nu\}$ be an increasing sequence of compact subsets of $E$ such that
$\mu_T$ and $\mu_{\partial T}$ are concentrated on $\cup_\nu K_\nu$,
and let $P_\nu$ be a finite rank projection of norm one such that $\|P_\nu x-x\|\leq 1/\nu$ for all $x\in K_\nu$. Thus $P_\nu x\to x$ as $\nu\to \infty$ for all $x\in \cup_\nu K_\nu$. %$x\in \supp T\cup \supp \partial T$.

Consider first the case when $\supp T$ is bounded.
Let $T_n:= P_{n\#} T$. Then $T_n\weakto T$ in the weak sense of currents.
In fact, for every $f\,d\pi\in D^k(E)$ with $\Lip\pi_i \leq 1$ for all $i=1,\ldots, k$ we have
\begin{align*}
    |T(f\circ P_n\, d\pi\circ P_n) & -T(f\,d\pi)|   \leq |T(f\circ P_n\, d\pi\circ P_n)-T(f\circ P_n\,d\pi)|  +\\
    &\qquad |T(f\circ P_n\, d\pi)-T(f\,d\pi)| \\
    &\leq \sum_{i=1}^k \int_E |f\circ P_n|\cdot |\pi_i\circ P_n- \pi_i|\, d\mu_{\partial T} + \\
    & \qquad \Lip\, f \sum_{i=1}^k\int_E |\pi_i\circ P_n- \pi_i|\, d\mu_T +\\
    &\qquad
    |T(f\circ P_n\, d\pi)-T(f\,d\pi)| \qquad \mbox{ by proposition~5.1 of~\cite{AmbrKirch00}}\\
    &
    \leq (\|f\|_\infty  +\Lip\, f ) k \int_E  \|P_n x-  x\|\, d(\mu_{\partial T} +\mu_T) +\\
    &\qquad
    |T(f\circ P_n\, d\pi)-T(f\,d\pi)|,
\end{align*}
all the terms in the right-hand side tending to zero as $n\to \infty$ by the choice of $P_n$ (the first one by Lebesgue
theorem, recalling that $\|P_n x-  x\|\leq 2\|x\|$ and the support of $T$, and hence of $\partial T$, is bounded, while the last term
because $f(P_n(x))\to f(x)$ for $\mu_T$-a.e. $x\in E$).

Further, we have $\MM(T_n)\leq \MM(T)$ which together with lower semicontinuity of the mass  with respect to weak convergence
gives $\MM(T_n)\to \MM(T)$, and the latter implies $\mu_{T_n}\rightharpoonup \mu_T$
in the narrow sense of measures
as $n\to \infty$.
In the same way one shows that
$\mu_{\partial T_n}\rightharpoonup \mu_{\partial T}$.

For the general case of a current $T$ with possibly unbounded support, we approximate
$T$ by a sequence $\{T_\nu\}\subset \M_k(E)$, such that each $T_\nu$ has bounded support and
$\MM(T_\nu-T)+\MM(\partial T_\nu -\partial T)\to 0$ as $\nu\to \infty$
(for this purpose just take $T_\nu := T\res g_\nu$ for a $g_\nu\in \Lip_1(E)$ with bounded support having
$0\leq g_\nu \leq 1$ and
$g_\nu=1$ on $B_\nu (0)$). Approximating now each $T_\nu$ by the currents $T_\nu^n$ as above, and choosing a diagonal
subsequence provided by Lemma~\ref{lm_currdiag1}, we get the result.
%Finally,
%\[
%\partial T_n=P_{n\#}(\partial T)=0,
%\]
%whenever $\partial T=0$.
\end{proof}

\begin{lemma}\label{lm_cycl1_massfin1}
Let $E$ be a finite-dimensional normed space
%$\dim E=n$ (i.e.\ $E=\R^n$ as a set)
endowed with the norm $\|\cdot\|$,
and
$T\in \M_1(E)$.
Then
\begin{equation}\label{eq_reprcurr1}
T(f\,d\pi)=\int_{E} f(x) (\nabla\pi(x),l(x))\, d\mu_T(x),
\end{equation}
when $\pi\in C^1(E)$,
for some Borel measurable vector field $l\colon E\to E$ satisfying $\|l(x)\|=1$
for $\mu_T$-a.e. $x\in \E$, where $(\cdot,\cdot)$ stands for the scalar product of vectors.
\end{lemma}

\begin{proof}
The representation of $T$ in the form~\eqref{eq_reprcurr1}
with $l\in L^\infty(E;\mu_T)$
is due to theorem~1.3 from~\cite{Williams10}
when $\mu_T\ll \mathcal{L}^n$; the
general case follows by approximating $T$ by a sequence of
$T_k\in \M_1(E)$ with
 $T_k\weakto T$, $\mu_{T_k}\weakto \mu_T$ as $k\to +\infty$, and $\mu_{T_k}\ll \mathcal{L}^n$ for all $k\in \N$.
Further, minding that $\|\nabla\pi(x)\|'\leq \Lip\, \pi $ for all $x\in E$,
where $\|\cdot\|'$ stands for the norm in the space $E'$ dual to $E$, the representation~\eqref{eq_reprcurr1} implies
\[
|T(f\,d\pi)|\leq \int_{E} f(x) \|\nabla\pi(x)\|'\cdot \|l(x)\|\, d\mu_T(x) \leq \Lip\pi \int_{E} f(x) \|l(x)\|\, d\mu_T(x),
\]
so that, by the definition of the mass measure of a metric current  one has
$\mu_T\leq \|l\|\mu_T$. This implies $\|l(x)\|\geq 1$
for $\mu_T$-a.e. $x\in \E$.
To prove the opposite inequality, let
$a\colon \R^n\to \R^n$ be a Borel measurable vector field with
$\|a(x)\|'=1$
such that $(a(x),l(x))=\|l(x)\|$
(such a vector field exists, say, in view of corollary~A.2.1 of~\cite{Vaeth97}).
Denote for the sake of brevity $\mu:=\|l\|\mu_T$.
For a given $\varepsilon>0$, we %first
choose a finite $\delta$-net $\{c_i\}_{i=1}^k$ of
the unit sphere $\{\|x\|'=1\}$, where $\delta=\varepsilon/\mu(E)$, and set
\begin{align*}
E_i& :=\{\|a(x)-c_i\|'\leq \delta\}\\
D_1& :=E_1, \qquad D_i:= E_i\setminus \cup_{j=1}^{i-1} D_i,
\end{align*}
so that for $a_\varepsilon :=\sum_{i=1}^k \mathbf{1}_{D_i} c_i$
one has
\begin{align*}
\int_E \left\|a(x)-a_\varepsilon(x)\right\|\, d\mu
 &=
\sum_{i=1}^k\int_{D_i} \left\|a(x)-c_i\right\|\, d\mu \\
& \leq \delta \sum_{i=1}^k\mu(D_i)= \delta\mu(E)\leq \varepsilon.
\end{align*}
Letting $\pi_i\colon E\to \R$ be a Lipschitz function with
$\Lip\pi_i=1$ and $\nabla \pi_i= c_i$,
one gets
\begin{align*}
\mu_T(\{l>1+\alpha\}) &\geq \sum_{i=1}^k T\left(\mathbf{1}_{\{l>1+\alpha\}}\mathbf{1}_{D_i}\,d\pi_i\right) \\
& =
\int_{\{l>1+\alpha\}}
 (a_\varepsilon,l(x))\, d\mu_T(x) \\
 &\geq \int_{\{l>1+\alpha\}}
 (a,l(x))\, d\mu_T(x) - \int_{E}
 \|a-a_\varepsilon\|'\cdot \|l(x)\|\, d\mu_T(x) \\
 &\geq (1+\alpha)\mu_T(\{l>1+\alpha\})-\varepsilon.
\end{align*}
Sending $\varepsilon\to 0^+$, we get
 $\mu_T(\{l>1+\alpha\})\geq (1+\alpha) \mu_T(\{l>1+\alpha\})$, which can be only true when
 $ \mu_T(\{l>1+\alpha\})=0$. Since $\alpha>0$ can be taken arbitrary, we get $\|l\|\leq 1$ which concludes the proof.
\end{proof}

Now we consider another construction which is used in the paper.
 Let $(E_i,d_i)$ be metric spaces, $i=1,2$,
 $T_1\in \M_1(E_1)$ and  $\mu_2\in \M_0(E_2)$.
 We define the current $T_1\times \mu_2\in \M_1(E_1\times E_2)$ by setting
\begin{equation}\label{eq_prodcurr1}
(T_1\times \mu_2)(\omega):=\int_{E_2} T_1 (\omega(\cdot, x_2))\, d\mu_2(x_2)
\end{equation}
for every $\omega=f\,d\pi\in D^1(E_1\times E_2)$.  Analogously we define
$\mu_1\times T_2\in \M_1(E_1\times E_2)$ for  $T_2\in \M_1(E_2)$ and  $\mu_1\in \M_0(E_1)$.

\begin{lemma}\label{lm_cycl1_Mprod1}
 Let $(E_i,d_i)$ be complete spaces, $T_i\in \mathcal{N}_1(E_i)$, $i=1,2$,
 and
\[
T:=T_1\times \mu_{T_2} +
\mu_{T_1}\times T_2\in \mathcal{N}_1(E_1\times E_2).
\]
Then $\MM(T)=\MM(T_1)\MM(T_2)$, if
the distance $d$ in $E_1\times E_2$ is defined by
\[
d((x_1, x_2),(x_1', x_2')):= d_1(x_1, x_1')\vee d_2(x_2, x_2').
\]
\end{lemma}

\begin{proof}
Let us first observe that it is enough to show
\begin{equation}\label{eq_cycl1_Mprod1main}
\MM(T)\leq\MM(T_1)\MM(T_2).
\end{equation}
In fact, denoting by $P_1\colon E_1\times E_2\to E_1$ the projection
map $P_1(x_1, x_2):= x_1$, we have
\[
P_{1\#} (T_1\times \mu_{T_2}) = \MM(T_2) T_1, \qquad P_{1\#} (\mu_{T_1}\times T_2)=0,
\]
so that
\[
\MM(T)\geq \MM(P_{1\#} T)= \MM(T_1)\MM(T_2).
\]
We divide the proof of the remaining claim~\eqref{eq_cycl1_Mprod1main} in three steps.

{\em Step 1}. Consider the case when $E_i$ are finite-dimensional normed spaces with norms $\|\cdot\|_i$.
Then $E_1\times E_2$ is equipped with the norm $\|(x_1,x_2)\|:=\|x_1\|_1\vee \|x_2\|_2$.
By Lemma~\ref{lm_cycl1_massfin1} we may assume
\[
T_i(f_i\,d\pi_i)=\int_{E_i} f_i(x_i) (\nabla\pi_i(x_i),l_i)\, d\mu_{T_i}(x_i), \qquad
\|l_i\|_{E_i}=1,
\]
for every $f_i\,d\pi_i\in D^1(E_i)$. Then, for $f\,d\pi\in D^1(E_1\times E_2)$ one has
\begin{align*}
T(f\,d\pi) &=\int_{E_2}\left(\int_{E_1} f(x_1, x_2) (\nabla\pi(x_1, x_2),l_1)\, d\mu_{T_1}(x_1)\right) d\mu_{T_2}(x_2)\\
& \quad + \int_{E_1}\left(\int_{E_2} f(x_1, x_2) (\nabla\pi(x_1, x_2),l_2)\, d\mu_{T_2}(x_2)\right) d\mu_{T_1}(x_1)\\
& = \int_{E_1\times E_2}f(x_1, x_2) (\nabla\pi(x_1, x_2),l)\, d\mu_{T_1}\otimes\mu_{T_2}(x_1,x_2),
\end{align*}
where $l:= (l_1,0)+ (0, l_2)=(l_1,l_2)\in E_1\times E_2$. Since $\|l\|=1$, we have by Lemma~\ref{lm_cycl1_massfin1}
that
$\mu_T=\mu_{T_1}\otimes \mu_{T_2}$, so that in particular
$\MM(T)=\MM(T_1)\MM(T_2)$.

{\em Step 2}.
We now show this result for the case when both $E_i$ are Banach spaces
with metric approximation property.
Let $T_i^n\in \M_1(E_i)$ be normal currents supported over some
finite-dimensional subspaces of $E_i$ such that $T^n_i\weakto T_i$ in the weak sense of currents, while
$\mu_{T_n^i}\weakto \mu_{T_i}$
 as $n\to \infty$ (such sequences of currents exist  due to Lemma~\ref{lm_currapprMAP1}).
We claim that for
\[
T^n:=T_1^n\times \mu_{T_2^n} +
\mu_{T_1^n}\times T_2^n\in \mathcal{N}_1(E_1\times E_2)
\]
one has $T^n\weakto T$ in the weak sense of currents
 as $n\to \infty$. This would complete the proof of this step since
 then
 \begin{align*}
 \MM(T) &\leq\liminf_n \MM(T^n) \\
 &= \liminf_n \MM(T^n_1) \MM(T^n_2) \qquad\qquad\qquad\qquad\mbox{(by Step 1)}\\
 & =\lim_n \MM(T^n_1) \lim_n \MM(T^n_2)= \MM(T_1) \MM(T_2).
 \end{align*}

 To show the claim consider an arbitrary $\omega=f\,d\pi\in D^1(E_1\times E_2)$.
 One has
 \begin{equation}\label{eq_cycl1T1T2a}
 \begin{aligned}
\int_{E_2} T_1^n (\omega(\cdot, x_2))\, d\mu_{T_2^n}(x_2) &= \int_{E_2} (T_1^n-T_1) (\omega(\cdot, x_2))\, d\mu_{T_2^n}(x_2)\\
&\qquad + \int_{E_2} T_1 (\omega(\cdot, x_2))\, d\mu_{T_2^n}(x_2).
 \end{aligned}
 \end{equation}
 But $(T_1^n-T_1) (\omega(\cdot, x_2))\to 0$ since $T_1^n\to T_1$, and  minding that
 \[
 |(T_1^n-T_1) (\omega(\cdot, x_2))|\leq (\MM(T_1^n)+\MM(T_1))\|f\|_\infty\Lip\pi\leq
 3\MM(T_1)\|f\|_\infty\Lip\pi
 \]
 when $n$ is sufficiently large, we get by means of Lebesgue dominated convergence theorem
 \[
 \int_{E_2} (T_1^n-T_1) (\omega(\cdot, x_2))\, d\mu_{T_2^n}(x_2)\to 0
 \]
 as $n\to +\infty$. On the other hand, the map
 $x_2\subset E_2\mapsto T_1 (\omega(\cdot, x_2))$ is bounded by $\MM(T_1)\|f\|_\infty\Lip\,\pi$ and
 continuous by the basic properties of currents, because $x_2^k\to x_2$ in $E_2$ implies
 $f(\cdot,x_2^k)\to f(\cdot,x_2)$, pointwise, and hence also in $\mu_{T_1}$ (because
 $\|f(\cdot, x_2^k)\|_\infty\leq \|f\|_\infty$,  and $\pi(\cdot,x_2^k)\to \pi(\cdot,x_2)$, pointwise with
 $\Lip\,\pi (\cdot, x_2^k)\leq \Lip\,\pi$). Therefore,
 \[
\int_{E_2} T_1 (\omega(\cdot, x_2))\, d\mu_{T_2^n}(x_2)\to
\int_{E_2} T_1 (\omega(\cdot, x_2))\, d\mu_{T_2}(x_2),
 \]

 since
 $\mu_{T_n^2}\weakto \mu_{T_2}$
 as $n\to \infty$.
 Thus, from~\eqref{eq_cycl1T1T2a} we get
 \[
 \int_{E_2} T_1^n (\omega(\cdot, x_2))\, d\mu_{T_2^n}(x_2) \to \int_{E_2} T_1 (\omega(\cdot, x_2))\, d\mu_{T_2}(x_2).
 \]
 Analogously we obtain
 \[
 \int_{E_1} T_2^n (\omega(x_1,\cdot))\, d\mu_{T_1^n}(x_1) \to \int_{E_1} T_2 (\omega(x_1,\cdot))\, d\mu_{T_1}(x_2),
 \]
 and hence the claim.

 {\em Step 3}.
 In view of lemma~5.5 from~\cite{PaoSte11-acycl} and of the previous step of the proof the result is proven
 in the case $E_1=E_2=\ell^\infty$.
 If $E_i$ are arbitrary complete metric spaces,
 we may assume without loss of generality that they be Polish (otherwise just
 take $\supp T_i$ in place of $E_i$).
 Denoting then by $j_i\colon E_i\to \ell^\infty$ the isometric imbeddings, and minding that
 $\mu_{j_{i\#}T_i} = j_{i\#}\mu_{T_i}$, we get that
 \[
 j_{1\#}T_1\times \mu_{j_{2\#}T_2} + \mu_{j_{1\#}T_1}\times j_{2\#}T_2= (j_1,j_2)_{\#} T.
 \]
 But then $\MM((j_1,j_2)_{\#} T)\leq \MM (j_{1\#}T_1)\MM(j_{2\#}T_2)=\MM (T_1)\MM(T_2)$, but since the map
 $(j_1,j_2)\colon E_1\times E_2\to \ell^\infty\times \ell^\infty$ is an isometry, then
 $\MM (T)=\MM((j_1,j_2)_{\#} T)$, and the proof is completed.
 \end{proof}

\section{Auxiliary lemmata from probability theory}

Here we collect some more or less folkloric statements (or something ``around'' mathematical folklore) from
abstract probability theory which are used in the paper.
We start with the following compactification result which is a variation on the theme of lemma~3.1.4 from~\cite{Stroock00}.

\begin{lemma}\label{lm_cycl1_sepremetr1}
Let $(E,d)$ be a separable metric space. Then there is a new distance $\tilde d\leq d$ over $E$ topologically equivalent to
$d$ such that $(E,\tilde d)$ is totally bounded. In particular,
denoting by $\tilde E$ the completion of $E$ with respect to $\tilde d$ we have
that $\tilde E$ is compact, while
the space $C(\tilde E)=C_b(\tilde E)$ is separable. Thus,
letting $C_u(E,\tilde d)$ to stand for the set of bounded functions
uniformly continuous over $E$ with respect to $\tilde d$, we get
the existence of a countable set $\{f_k\}\subset C_u(E,\tilde d)$ dense
in  $C_u(E,\tilde d)$ in the uniform norm $\|\cdot\|_\infty$.
\end{lemma}

\begin{proof}
Let $\{x_k\}\subset E$ stand for a countable dense set in $E$, and consider the map
$g\colon E\to [0,1]^\N$
defined by
\[
g_n(x):= \frac{d(x,x_n)}{1+d(x,x_n)}
\]
for all $x\in E$. Note that $[0,1]^\N$ is compact when equipped with the product topology, while
the latter may be metrized, say, by the distance
\[
\hat d(x,y):= \sum_{k=1}^\infty \frac{|x_k-y_k|}{2^k}.
\]
Thus, defining
\[
\tilde d(x,y):= \hat d(g(x),g(y))
\]
for all $\{x,y\}\subset E\times E$, we get that $(E,\tilde d)$ is totally bounded.

Now, clearly,
\begin{align*}
\tilde d(x,y) & = \sum_{k=1}^\infty \frac 1 {2^k} \left|\frac{d(x,x_k)}{1+d(x,x_k)} - \frac{d(y,x_k)}{1+d(y,x_k)} \right|\\
& \leq  \sum_{k=1}^\infty \frac 1 {2^k} \left|d(x,x_k) - d(y,x_k) \right| \leq  \sum_{k=1}^\infty \frac 1 {2^k} d(x,y) = d(x,y).
\end{align*}
Vice versa,
 $\tilde d(y_j,y)\to 0$ for some $y\in E$ %again
 implies
$g_n(y_j)\to g_n(y)$, and hence
$d(y_j,x_n)\to d(y,x_n)$ for all $n\in \N$ as $j\to \infty$. Then
\[
d(y_j,y)\leq d(y_j,x_n) +d(y,x_n) \to 2 d(y,x_n),
\]
and hence, since $x_n$ is an arbitrary element of a dense set in $E$,
we get $d(y_j,y)\to 0$
as $j\to \infty$.

Hence, denoting by $\tilde E$ the completion of $E$ with respect to $\tilde d$ we have
that $\tilde E$ is compact, and therefore
the space $C(\tilde E)=C_b(\tilde E)$ is separable. Further, every $f\in C_b(\tilde E)$ is clearly uniformly continuous.
Vice versa, if $f\in C_u(E,\tilde d)$, then for every fundamental sequence
$\{y_j\}\subset E$ one has that $\{f(y_j)\}\subset \R$ is fundamental, and hence $f$ can be extended by continuity
to a function from $C(\tilde E)$. Thus we may identify $C_u(E,\tilde d)$ with $C(\tilde E)$, so that the last claim of
the lemma being proven is just separability of $C(\tilde E)$.
\end{proof}

\begin{remark}\label{rm_cycl1_comp1}
For the case of a Euclidean space $E:=\R^n$ (or, more generally, for a space with Heine-Borel property, i.e.\ a space
where closed balls are compact) the above Lemma~\ref{lm_cycl1_sepremetr1} gives just the ordinary Alexandrov
one-point compactification $\tilde{E}$. In fact, if a sequence $\{y_k\}\subset E$ is fundamental with respect to the new distance
$\tilde d$, then so is the sequence
$\{d(y_k, x_n)/(1+d(y_k, x_n))\}\subset \R$ for each $n\in \N$. Then either of the following two separate cases may happen.
\begin{itemize}
\item[(i)] $d(y_k, x_n)/(1+d(y_k, x_n))\to 1$, which means $d(y_k, x_n)\to \infty$ for some $n\in \N$, which happens if and only if $y_k\to \infty$ (i.e.\ $d(y_k,y)\to \infty$ for all $y\in E$) as $k\to \infty$. This is the case when the $\{y_k\}$ determines the point $\infty\in \tilde E$.
\item[(ii)] The sequence $\{y_k\}$ is uniformly bounded (note that the case of $y_k\to \infty$ as $k\to \infty$ only for a subsequence of $\{y_k\}$ is excluded since otherwise  one would have $d(y_k, x_n)/(1+d(y_k, x_n))\to 1$ along this subsequence, and hence for the whole sequence, since the latter sequence of numbers is fundamental). Then up to a subsequence (not relabeled) $y_k\to y\in E$, hence
    $d(y_k,x_n)\to d(y, x_n)$, and therefore also $d(y_k, x_n)/(1+d(y_k, x_n))\to d(y, x_n)/(1+d(y, x_n))$
     for all $n\in \N$
     as $k\to \infty$. Again, the latter convergence must be now valid for the whole original sequence, which means that the same must be true also for convergence
     $d(y_k,x_n)\to d(y, x_n)$.
     Now, for any other convergent subsequence of $\{y_k\}$ (again not relabeled), say,
     $y_k\to z\in E$, one would have $d(y_k,x_n)\to d(z, x_n)$, which implies $d(y, x_n)=d(z, x_n)$ for all $n\in \N$.
     This means $y=z$ and hence  the whole sequence $\{y_k\}$ is convergent to $y\in E$.
\end{itemize}
Summing up, we have $\tilde E= E\cup \{\infty\}$.
Therefore, $C_u(E,\tilde d)$ consists of continuous (with respect to $d$) functions having (finite) limits at infinity.
\end{remark}

\begin{remark}\label{rm_cycl1_comp2}
In the case of a Euclidean space $E:=\R^n$ for every compact $K\subset E$
there is a $C>0$ such that
\[
d(y,z)\leq C \tilde d(y,z) \mbox{ for all } (y,z)\in K\times K.
\]
To show this suppose the contrary, i.e.\ the existence of $\{(y_k,z_k)\}\subset K\times K$ such that
\[
\lim_k \frac{\tilde d(y_k,z_k)}{d(y_k,z_k)}\to 0,
\]
and thus
\[
\lim_k \left|
\frac{d(y_k,x_n)}{1+ d(y_k,x_n)} - \frac{d(z_k,x_n)}{1+ d(z_k,x_n)}
\right| \frac{1}{d(y_k,z_k)}
= 0,
\]
for all $n\in \N$,
which is only possible when
\begin{equation}\label{eq_cycl1_xnyn1}
\lim_k
\frac{|d(y_k,x_n)-d(z_k,x_n)|}{d(y_k,z_k)} = 0,
\end{equation}
for all $n\in \N$. Since by compactness of $K$ we may assume without loss of generality that $z_k\to z$ and $y_k\to y$ as $k\to \infty$, then the above equality is only possible once
and in particular $y=z$. But
for $x_n\neq y$ the relationship
\[
\frac{d(y_k,x_n)-d(z_k,x_n)}{d(y_k,z_k)} = \frac{y-x_n}{|y-x_n|} \cdot \frac{y_k-z_k}{|y_k-z_k|} + o(1)
\]
for $k\to \infty$ holds.
Since up to a subsequence (not relabeled) $(y_k-z_k)/|y_k-z_k|\to e$ as $k\to \infty$ for some unit vector $e$, then choosing an
$n\in \N$ such that
\[
\frac{y-x_n}{|y-x_n|} \cdot e >0,
\]
we get a contradiction with~\eqref{eq_cycl1_xnyn1}.
\end{remark}

We also use in the paper the following easy consequence of corollary~7.7.2 from~\cite{Bogachev06}
(i.e.\ of the Kolmogorov extension theorem).

\begin{lemma}\label{lm_cycl1_probext1}
Let $(X,\Sigma)$ be a measure space ($X$ being a metric space and $\Sigma$ being its Borel $\sigma$-algebra) and $\eta_k$ be
(Borel) tight probability measures over $X^k$ satisfying the following compatibility conditions:
\begin{align*}
    \pi_{k-1\#} \eta_k & =\eta_{k-1},\\
    \pi^{k-1}_{\#} \eta_k & =\eta_{k-1},
\end{align*}
where $\pi_{k-1}\colon X^k= X^{k-1}\times X\to X^{k-1}$ and
$\pi^{k-1}\colon X^k= X\times X^{k-1}\to X^{k-1}$ are defined by
\begin{align*}
    \pi_{k-1} (x_1,\ldots, x_{k-1}, x_k) & :=(x_1,\ldots, x_{k-1}),\\
    \pi^{k-1} (x_1, x_2,\ldots,  x_k) & := (x_2,\ldots, x_k).
\end{align*}
Then there is a probability measure $\eta_*$ over $X^\Z$ such that
\[
\eta_*\left(
\bigcap_{j=k}^l p_j^{-1} (e_j)
\right)= \eta_{l-k}
\left(
\prod_{j=k}^l e_j
\right),\qquad e_j\in \Sigma,
\]
where $p_j\colon X^\Z\to X$ is defined by $p_j(x):=(x)_j$.
\end{lemma}

%\bibliographystyle{plain}
%\bibliography{mathopt}

\begin{thebibliography}{10}

\bibitem{AmbrKirch00}
L.~Ambrosio and B.~Kirchheim.
\newblock Currents in metric spaces.
\newblock {\em Acta Math.}, 185(1):1--80, 2000.

\bibitem{Bang99}
V.~Bangert.
\newblock Minimal measures and minimizing closed normal one-currents.
\newblock {\em Geom. Funct. Anal.}, 9(3):413--427, 1999.

\bibitem{BishDeLeeuw59}
E.~Bishop and K.~de~Leeuw.
\newblock The representations of linear functionals by measures on sets of
  extreme points.
\newblock {\em Ann. Inst. Fourier. Grenoble}, 9:305--331, 1959.

\bibitem{Bogachev06}
V.~Bogachev.
\newblock {\em Measure Theory, vol.~I and II}.
\newblock Springer-Verlag, 2006.

\bibitem{DePasGelGran06}
L.~De~Pascale, M.S. Gelli, and L.~Granieri.
\newblock Minimal measures, one-dimensional currents and the
  {M}onge-{K}antorovich problem.
\newblock {\em Calc. Var. Partial Differential Equations}, 27(1):1--23, 2006.

\bibitem{MunMar09-II}
V.~Mu{\~n}oz and R.~P{\'e}rez~Marco.
\newblock Ergodic solenoidal homology. {II}. {D}ensity of ergodic solenoids.
\newblock {\em Aust. J. Math. Anal. Appl.}, 6(1):Art. 11, 8, 2009.

\bibitem{MunMar09sc}
V.~Mu{\~n}oz and R.~P{\'e}rez~Marco.
\newblock {S}chwartzman cycles and ergodic solenoids.
\newblock {\em http://arxiv.org/abs/0910.2837}, 2009.

\bibitem{MunMar11b}
V.~Mu{\~n}oz and R.~P{\'e}rez~Marco.
\newblock Ergodic solenoidal homology: realization theorem.
\newblock {\em Comm. Math. Phys.}, 302(3):737--753, 2011.

\bibitem{MunMar11a}
V.~Mu{\~n}oz and R.~P{\'e}rez~Marco.
\newblock Ergodic solenoids and generalized currents.
\newblock {\em Rev. Mat. Complut.}, 24(2):493--525, 2011.

\bibitem{PaoSte11-acycl}
E.~Paolini and E.~Stepanov.
\newblock Decomposition of acyclic normal currents in a metric space.
\newblock {\em J. Funct. Anal.}, 263(11):3358--3390, 2012.

\bibitem{Schwartz57}
S.~Schwartzman.
\newblock Asymptotic cycles.
\newblock {\em Ann. of Math. (2)}, 66:270--284, 1957.

\bibitem{Schwartz97}
S.~Schwartzman.
\newblock Asymptotic cycles on non-compact spaces.
\newblock {\em Bull. London Math. Soc.}, 29(3):350--352, 1997.

\bibitem{Smirnov94}
S.K. Smirnov.
\newblock Decomposition of solenoidal vector charges into elementary solenoids
  and the structure of normal one-dimensional currents.
\newblock {\em St. Petersburg Math. J.}, 5(4):841--867, 1994.

\bibitem{Stroock00}
D.W. Stroock.
\newblock {\em Probability Theory, an Analytic View}.
\newblock Cambridge University Press, 2000.

\bibitem{Vaeth97}
M.~V{\"{a}}th.
\newblock {\em Ideal spaces}, volume 1664 of {\em Lecture Notes in
  Mathematics}.
\newblock Springer-Verlag, N.Y., 1997.

\bibitem{Williams10}
M.~Williams.
\newblock Metric currents, differentiable structures, and {C}arnot groups.
%\newblock Preprint. {\tt http://arxiv.org/abs/1008.4120}, 2010.
\newblock {\em Ann. Sc. Norm. Super. Pisa, Cl. Sci.}, 11(5), no.~2:259--302, 2012.

\end{thebibliography}

\end{document}